\documentclass{amsart}
\usepackage{mathrsfs,amssymb,multirow,setspace}
\usepackage[margin=2cm]{geometry}
\theoremstyle{plain}
\newtheorem{theorem}{Theorem}[section]
\newtheorem{lemma}[theorem]{Lemma}
\newtheorem{proposition}[theorem]{Proposition}
\newtheorem{corollary}[theorem]{Corollary}
\newtheorem{conjecture}[theorem]{Conjecture}
\theoremstyle{definition}
\newtheorem{notation}[theorem]{Notations}

\newtheorem{definition}[theorem]{Definition}

\theoremstyle{remark}

\input xy
\xyoption{all}

\def\G{\Gamma}
\def\vt{\vartheta}

\begin{document}


\title[Rank Properties of ${\rm End}(B_n)$]{Rank Properties of the Semigroup of Endomorphisms over Brandt semigroup}
\author[Jitender Kumar]{Jitender Kumar}
\address{Department of Mathematics, Birla Institute of Technology and Science Pilani, Pilani, India}
\email{jitenderarora09@gmail.com}


\begin{abstract}
Howie and Ribeiro \cite{a.Howie99,a.Howie00} introduced various ranks, viz. small rank, lower rank, intermediate rank, upper rank and the large rank of a finite semigroup. In this note, we investigate all these ranks of the semigroup of endomorphisms over Brandt semigroup.
\end{abstract}



\maketitle


\section*{Introduction and Preliminaries}

Since the work of Marczewski \cite{a.marczewski66}, many authors have studied the rank properties in the context of general algebras (cf. \cite{a.camron02,a.howie87,a.howie92,a.jk16-2,t.jdm02,a.ruskuc94,a.ping11}). The concept of rank for general algebras is analogous to the concept of dimension in linear algebra. The dimension of a vector space is the maximum cardinality of an independent subset, or equivalently, it is the minimum cardinality of a generating set of the vector space.
A subset $U$ of a semigroup $\G$ is said to be \emph{independent} if every element of $U$ is not in the subsemigroup generated by the remaining elements of $U$, i.e. \[ \forall a \in U, \; a \notin \langle U \setminus \{a\} \rangle .\] It can be observed that the minimum size of a generating set need not be equal to the maximum size of an  independent set in a semigroup. Accordingly, Howie and Ribeiro have considered various concepts of ranks for a finite semigroup $\G$ (cf. \cite{a.Howie99,a.Howie00}).
\begin{enumerate}
\item $r_1(\G) = \max\{k :$ every subset $U$ of cardinality $k$ in $\G$ is independent\}.
\item $r_2(\G) = \min\{|U| : U \subseteq \G, \langle U\rangle = \G\}$.
\item $r_3(\G) = \max\{|U| : U \subseteq \G, \langle U\rangle = \G, U$ is independent\}.
\item $r_4(\G) = \max\{|U| : U \subseteq \G, U$  is independent\}.
\item $r_5(\G) = \min\{k :$ every subset $U$ of cardinality $k$ in $\G$ generates $\G$\}.
\end{enumerate}
It can be observed that \[r_1(\G) \le r_2(\G) \le r_3(\G) \le r_4(\G) \le r_5(\G).\] Thus,
$r_1(\G), r_2(\G), r_3(\G), r_4(\G)$ and $r_5(\G)$ are, respectively, known as \emph{small rank}, \emph{lower rank}, \emph{intermediate rank}, \emph{upper rank} and \emph{large rank} of $\G$. In this note, we obtain all these ranks except upper rank of the semigroup ${\rm End}(B_n)$ (with respect to composition of mappings) of endomorphisms over a Brandt semigroup. A lower bound of $r_4({\rm End}(B_n))$ is obtained. First, we recall the necessary background material and fix our notation. It is well known that the set $\{(1 \; 2), (1\; 2 \cdots n)\}$ of permutations is a generating of the minimum cardinality in $S_n$, the symmetric group of degree $n$. Hence, for $n \ge 3$, we have $r_2(S_n) = 2$. It is immediately evident that the set of
transpositions $\{(1 \; 2), (2 \; 3), \ldots,(n-1 \; n)\}$ forms an independent generating set of $S_n$ so that $r_3(S_n) \ge n-1$. Further, the intermediate rank and the upper rank of the symmetric group $S_n$ is obtained in the following theorem.

\begin{theorem}[{\cite[Theorem 1]{a.whiston00}}]\label{int. n upper S_n}
Given an independent set inside $S_n$, the size of the set is at
most $n - 1$, with equality only if the set generates the whole group $S_n$. Hence, For $n \ge 2$, we have $r_3(S_n) = r_4(S_n) = n - 1.$
\end{theorem}

\begin{definition}\label{d.Bn}
For any integer $n \ge 1$, let $[n] = \{1, 2, \ldots, n\}$. The semigroup $(B_n, +)$, where $B_n = ([n]\times[n])\cup \{\vartheta\}$ and the operation $+$ is given by
\[ (i,j) + (k,l) =
                \left\{\begin{array}{cl}
                (i,l) & \text {if $j = k$;}  \\
                \vartheta     & \text {if $j \neq k $}
                  \end{array}\right.  \]
and, for all $\alpha \in B_n$, $\alpha + \vartheta = \vartheta + \alpha = \vartheta$, is known as \emph{Brandt semigroup}. Note that $\vartheta$ is the (two sided) zero element in $B_n$. For more details on Brandt semigroup, we refer the reader to \cite{b.Howie95}.
\end{definition}

\begin{notation}

(1) The set of idempotent elements of $B_n$ is denoted by $I(B_n)$. Note that $I(B_n) = \{(i, i): i \in [n]\} \cup \{\vartheta\}.$

(2) For $\alpha \in B_n$, $\xi_\alpha$ denotes the constant map on $B_n$ that sends all the elements of $B_n$ to $\alpha$ i.e. $x \xi_\alpha  = \alpha \;\;\forall x \in B_n$. The set of constant maps on $B_n$ with idempotent value is denoted by $\mathcal{C}_{I(B_n)}$ i.e. $\mathcal{C}_{I(B_n)} = \{\xi_\alpha: \alpha \in I(B_n)\}.$

(3) The set of automorphisms and endomorphisms on $B_n$ is denoted by ${\rm Aut}(B_n)$ and ${\rm End}(B_n)$, respectively.
\end{notation}

Now, we recall the results concerning to the elements of the semigroup ${\rm End}(B_n)$ from \cite{a.jk14-1}.

\begin{proposition}[{\cite[Proposition 2.1]{a.jk14-1}}]\label{sn-iso-autbn}
The assignment $\sigma \mapsto \phi_\sigma: S_n \rightarrow {\rm Aut}(B_n)$ is an isomorphism, where the mapping $\phi_\sigma: B_n \rightarrow B_n$ is given by, $\forall i, j \in [n]$, \[(i, j)\phi_\sigma = (i\sigma, j\sigma)\; \mbox{ and }\; \vt\phi_\sigma = \vt.\]
\end{proposition}

\begin{theorem}[{\cite[Theorem 2.2]{a.jk14-1}}]\label{End(B_n)}
For $n \ge 1$, we have ${\rm End}(B_n) = {\rm Aut}(B_n) \cup \mathcal{C}_{I(B_n)}$.
\end{theorem}

In what follows, ${\rm End}(B_n)$ denotes the semigroup $({\rm End}(B_n), \circ)$. Further, the composition of mappings $f$ and $g$ will simply be denoted by $fg$.

\section{Main results}
In this section, we investigate all the ranks of the semigroup ${\rm End}(B_n)$. It can be easily observed that ${\rm End}(B_1)$ is an independent set and none of its proper subsets generate ${\rm End}(B_1)$. Hence, for $1 \le i \le 5$, we have \[r_i({\rm End}(B_1))  = |{\rm End}(B_1)| = 3.\] In the rest of the section, we shall investigate the ranks of the semigroup ${\rm End}(B_n)$, for $n > 1$. The small rank of ${\rm End}(B_n)$ comes as a consequence of the following result due to Howie and Ribeiro.

\begin{theorem}[{\cite[Theorem 2]{a.Howie00}}]\label{r1-gm}
Let $\Gamma$ be a finite semigroup, with $|\Gamma| \ge 2$. If $\Gamma$ is not a band, then $r_1(\Gamma) = 1$.
\end{theorem}

Owing to the fact that ${\rm End}(B_n)$ (for $n \ge 2$) has some non idempotent elements, it is not a band. For instance, the automorphism $\phi_\sigma$ (cf. Proposition \ref{sn-iso-autbn}), where $\sigma$ is a transposition in $S_n$,  is not an idempotent element in ${\rm End}(B_n)$. Hence, we have the following corollary of Theorem \ref{r1-gm}.

\begin{corollary}
For $n \ge 2$, $r_1({\rm End}(B_n)) = 1$.
\end{corollary}

In order to determine the lower rank of the semigroup ${\rm End}(B_n)$, the following lemma is useful in the sequel.

\begin{lemma}\label{gen-set-prop}
For $1 \le i \le k$, let $f, g_i \in {\rm End}(B_n)$ such that $f = g_1g_2\cdots g_k$. Then
\begin{enumerate}
\item $f \in {\rm Aut}(B_n)$ if and only if $g_i \in {\rm Aut}(B_n)$ for all i.
\item If $f= \xi_\vt$, then $g_j = \xi_\vt$ for some $j$.
\item $f \in \mathcal{C}_{I(B_n)} \setminus \{\xi_\vt\}$ if and only if  $g_j \in \mathcal{C}_{I(B_n)} \setminus \{\xi_\vt\}$ for some $j$.
\end{enumerate}
\end{lemma}

\begin{proof}

(1) Straightforward.

(2) Suppose $g_i \ne \xi_\vt$ for all $i$. Then, either $f \in {\rm Aut}(B_n)$ or $f$ is a non-zero constant map; a contradiction. Hence the result.

(3) Let $f \in \mathcal{C}_{I(B_n)} \setminus \{\xi_\vt\}$. Suppose $g_i \not\in \mathcal{C}_{I(B_n)} \setminus \{\xi_\vt\}$ for all $i$. If $g_j = \xi_\vt$ for some $j$, then $f = \xi_\vt$; a contradiction. Thus, $g_i \in {\rm Aut}(B_n)$ for all $i$ so that $f \in {\rm Aut}(B_n)$; again a contradiction. Consequently, $g_j \in \mathcal{C}_{I(B_n)} \setminus \{\xi_\vt\}$ for some $j$. The verification for the converse part is straightforward.
\end{proof}


\begin{corollary}\label{gen.two.const}
Any generating subset of ${\rm End}(B_n)$ must contain the zero map and a nonzero constant map.
\end{corollary}

\begin{theorem}\label{r2-endbn}
For $n \ge 3$, $r_2({\rm End}(B_n)) = 4$.
\end{theorem}

\begin{proof}
We show that the set $\mathcal{S} = \{\phi_{(1 \; 2)}, \phi_{(1 \; 2 \; \cdots \; n)}, \xi_{(1, 1)}, \xi_\vt \} $ of four elements is a generating set of the minimum cardinality in ${\rm End}(B_n)$ so that the result follows.
Let $f \in {\rm End}(B_n)$. By Theorem \ref{End(B_n)}, if $f \in {\rm Aut}(B_n)$, then $f = \phi_\sigma$ for some $\sigma \in S_n$ (cf. Proposition \ref{sn-iso-autbn}). Since $\sigma \in \langle (1 \; 2), (1 \; 2 \; \cdots \; n) \rangle$, we have $\phi_\sigma  \in \langle \phi_{(1 \; 2)}, \phi_{(1 \; 2 \; \cdots \; n)} \rangle$ so that $f \in \langle \mathcal{S} \rangle$. Thus, ${\rm Aut}(B_n) \subset \langle \mathcal{S} \rangle$.
If $f = \xi_{(i,i)}$, then for $\sigma \in S_n$ such that $1 \sigma = i$, we have $\xi_{(i, i)} = \xi_{(1, 1)} \phi_\sigma$. Consequently, $f \in \langle \mathcal{S} \rangle$ and $\mathcal{C}_{I(B_n)} \subset \langle \mathcal{S} \rangle$ so that ${\rm End}(B_n) = \langle \mathcal{S} \rangle$.

Let $V$ be any generating subset of ${\rm End}(B_n)$. For $n \ge 3$, we have $r_2(S_n) = 2$. Since ${\rm Aut}(B_n)$ is isomorphic to $S_n$ (cf. Proposition \ref{sn-iso-autbn}), $V$ must contain atleast two automorphisms over $B_n$ to generate ${\rm Aut}(B_n)$. By Corollary \ref{gen.two.const}, $V$ must contain atleast two more elements of ${\rm End}(B_n)$, the zero map and a nonzero constant map, to generate $\mathcal{C}_{I(B_n)}$. Thus, $|V| \ge 4$. Hence, $\mathcal{S}$ is a generating set of the minimum size.
\end{proof}

In the similar lines of the proof of the Theorem \ref{r2-endbn}, it is easy to prove that the set $\{\phi_{(1 \; 2)}, \xi_{(1, 1)}, \xi_\vt\}$ is a generating subset of the minimum cardinality in ${\rm End}(B_2)$. Hence, we have the following theorem.

\begin{theorem}
$r_2({\rm End}(B_2)) = 3.$
\end{theorem}

After ascertaining certain relevant properties of independent generating sets of ${\rm End}(B_n)$, we obtain the intermediate rank of the semigroup ${\rm End}(B_n)$.

\begin{lemma}\label{prop-ind-gen-set}
Let $U$ be an independent generating subset of ${\rm End}(B_n)$. Then $|U| \le n + 1.$
\end{lemma}

\begin{proof}
Since $U$ is a generating subset of ${\rm End}(B_n)$, by Corollary \ref{gen.two.const}, $U$ must contain two constant maps of ${\rm End}(B_n)$ viz. the zero map and exactly one nonzero constant map. For $i, j \in [n]$  with $i \ne j$, let if possible, $f = \xi_{(i, i)}$ and $g = \xi_{(j, j)} \in U$. Then, for $\sigma \in S_n$ such that $j \sigma = i$, note that $\xi_{(i, i)} = \xi_{(j, j)}\phi_\sigma.$ Since $\phi_\sigma \in \langle U \setminus \{f\} \rangle$, we have $f \in \langle U \setminus \{f\} \rangle$; a contradiction to $U$ is an independent set. Since ${\rm Aut}(B_n) \subset {\rm End}(B_n)$ and ${\rm Aut}(B_n) \simeq S_n$, by Theorem \ref{int. n upper S_n}, $U$ can contain atmost $n-1$ automorphisms. Hence, $|U| \le n + 1$.
\end{proof}

%
%
%
%

\begin{theorem}\label{r3}
For $n \ge 2$, $r_3({\rm End}(B_n)) = n + 1$.
\end{theorem}

\begin{proof}
In view of Lemma \ref{prop-ind-gen-set}, we show that the set \[\mathcal{T} = \{\phi_{(1 \; 2)}, \phi_{(2 \; 3)}, \ldots, \phi_{(n-1 \; n)}\} \cup \{\xi_{(1, 1)}, \xi_{\vt}\}\] of size $n + 1$, is an independent generating set so that the result follows.

Let $f \in {\rm End}(B_n)$. If $f \in {\rm Aut}(B_n)$, then $f = \phi_\sigma$ for some $\sigma \in S_n$ (cf. Proposition \ref{sn-iso-autbn}). Since $\sigma \in \langle (1 \; 2), (2 \; 3), \ldots,(n-1 \; n) \rangle $, we have $\phi_\sigma \in \langle \phi_{(1 \; 2)}, \phi_{(2 \; 3)}, \ldots, \phi_{(n-1 \; n)} \rangle$ so that ${\rm Aut}(B_n) \subset \langle T \rangle$. If $f = \xi_{(i, i)} (i \ne 1)$, then for $\sigma \in S_n$ such that $1\sigma = i$, we have $\xi_{(i, i)} = \xi_{(1, 1)}\phi_\sigma$ so that $\mathcal{C}_{I(B_n)} \subset \langle \mathcal{T} \rangle$. Consequently, ${\rm End}(B_n) = \langle \mathcal{T} \rangle$. Since the set $\{(1 \; 2), (2 \; 3), \ldots,(n-1 \; n)\}$ is an independent set in $S_n$, by Proposition \ref{sn-iso-autbn}, the set $\{ \phi_{(1 \; 2)}, \phi_{(2 \; 3)}, \ldots, \phi_{(n-1 \; n)} \}$ is an independent set in $\mathcal{T}$. For $f = \xi_{(1,1)}$ or $\xi_{\vt}$, by Lemma \ref{gen-set-prop}, we have $f \not\in \langle \mathcal{T} \setminus \{f\}\rangle$. Hence the result.
\end{proof}

Though Theorem \ref{r3} gives us a lower bound for upper rank of the semigroup ${\rm End}(B_n)$, in the following theorem we provide a better lower bound for $r_4({\rm End}(B_n))$.

\begin{theorem}
For $n \ge 2$, we have $r_4({\rm End}(B_n)) \ge n + 2$.
\end{theorem}

\begin{proof}
We prove that the set $U = \{\phi_{id}\} \cup \mathcal{C}_{I(B_n)}$ is an independent subset of ${\rm End}(B_n)$. Let $f \in U$. Since a permutation can not be generated from non-permutations, thus if $f = \phi_{id}$, then clearly $f \not\in \langle U \setminus \{f\} \rangle$. If $f = \xi_{\vt}$, then by Lemma \ref{gen-set-prop}(2), $f \not\in \langle U \setminus \{f\} \rangle$. We may now suppose that $f = \xi_{(i, i)}$. Note that for $f_j (1 \le j \le s) \in U \setminus \{f\}$, we have $\xi_{(i, i)} \ne f_1 f_2 \cdots f_s$. Thus, $f \not\in \langle U \setminus \{f\} \rangle$. Consequently, $U$ is an independent subset of size $n + 2$ so that $r_4({\rm End}(B_n)) \ge n + 2$.
\end{proof}


Based on our observations, we give the following conjecture for the upper rank of  ${\rm End}(B_n)$.

\begin{conjecture}
For $n \ge 2$, we have $r_4({\rm End}(B_n)) = n + 2.$
\end{conjecture}

We leave the proof, or a refutation of this conjecture as an open question. \\

To find the large rank of the semigroup, we adopt the technique introduced in \cite{a.jk14-2}. The technique relies on the concept of prime subsets of semigroups.  A nonempty subset $U$ of a semigroup $\G$ is said to be \emph{prime} if, for all $a, b \in \G$, \[ab \in U \mbox{ implies }  a \in U  \mbox{ or } b \in U.\]

\begin{theorem}[{\cite[Corollary 2.3]{a.jk14-2}}]\label{gm-lsgp}
If $V$ is a smallest prime subset of a finite semigroup $\G$, then $r_5(\G) = |\G| - |V| + 1$.
\end{theorem}

In view of Lemma $\ref{gen-set-prop}(2)$, note that the set $\{\xi_\vt\}$ is a smallest prime subset of ${\rm End}(B_n)$. Since $|{\rm End}(B_n)| = n! + n + 1$ (cf. Proposition \ref{sn-iso-autbn} and Theorem \ref{End(B_n)}), we have the large rank of the semigroup ${\rm End}(B_n)$ in the following theorem.

\begin{theorem}
For $n \ge 2$, $r_5({\rm End}(B_n)) = n! + n +1$.
\end{theorem}

\section*{Acknowledgement}
The author cordially thanks Professor M. V. Volkov and the anonymous referee for their helpful comments/suggestions on an earlier version of this paper. The author wishes to acknowledge the support of research initiation grant [0076/2016] funded by BITS Pilani, Pilani.

\bibliographystyle{abbrv}

\end{document}